\documentclass[12pt]{article}
\usepackage[centertags]{amsmath}
\usepackage{amsfonts}
\usepackage{color,graphicx,shortvrb}
\usepackage{amssymb}
\usepackage{amsthm}
\usepackage{psfrag}
\usepackage{empheq}
\usepackage{newlfont}
\usepackage{epstopdf}
\usepackage{subfigure}

\newtheorem{theorem}{Theorem }
\newtheorem{remark}{Remark }
\newtheorem{lemma}{Lemma }

\newtheorem{definition}{Definition}
\newtheorem{proposition}{Proposition}

\begin{document}
\normalsize

\title{Linearized stability analysis of Caputo-Katugampola fractional-order nonlinear systems}
\author{{\small D. Boucenna \thanks{Laboratory of Advanced Materials,
Faculty of Sciences, Badji Mokhtar-Annaba University, P.O. Box 12,
Annaba, 23000, Algeria}, \ A. Ben
Makhlouf \thanks{Department of Mathematics, College of Science, Jouf University, Sakaka, Saudi Arabia} \thanks{Department of Mathematics, Faculty of Sciences of
Sfax, Route Soukra, BP 1171, 3000 Sfax, Tunisia}, \ O. Naifar
\thanks{CEM Lab, Department of Electrical Engineering, National
School of Engineering, University of Sfax, Tunisia}, \ A.
Guezane-Lakoud
\thanks{Laboratory of Advanced Materials, Faculty of Sciences,
Badji Mokhtar-Annaba University, P.O. Box 12, Annaba, 23000,
Algeria}, \ M. A. Hammami
\thanks{Department of Mathematics, Faculty of Sciences of
Sfax, Route Soukra, BP 1171, 3000 Sfax, Tunisia}
 }}
\date {}

\maketitle
\begin{abstract}
In this paper, a linearized asymptotic stability result for a
Caputo-Katugampola fractional-order systems is described. An application
is given to demonstrate the validity of the proposed
results.\end{abstract}



\section{Introduction}
The integer-order calculus is inconvenient for several physical
systems whose true dynamics contain fractional (non-integer)
derivatives. To accurately model these systems, the
fractional-order differential equations are used. For example heat
transfer systems \cite{2}, financial systems \cite{3} and
electromagnetic systems \cite{1}, have been modelled using the
fractional-order calculus. In the last years, the use of
fractional-order equations in the stability theory has distinctly
risen \cite{art1,6,8,3,art2,4,5,12,11,10,9}, and several works
have been done in this context, we cite for example two main
methods. Firstly, the Lyapunov's first method: as the method of
linearization, the study of the linear equation by means of
Lyapunov exponents and the theorem of linearized asymptotic
stability \cite{8} or \cite{expon}. Secondly, Lyapunov's Second
Method: as the method of Lyapunov functions, for instance,
see \cite{40} or \cite{15}.\\
In this paper, we propose a Lyapunov's first method for the zero solution of the following Caputo-Katugampola fractional-order system of order $\alpha \in (0,1)$, $\rho> 0$:

\begin{equation}
^{C}D_{t_0^{+}}^{\alpha ,\rho }x\left( t\right) =A x(t)+
f\left(x\left( t\right) \right),\ t \geq t_0, \label{1.1s}
\end{equation}%
where $A\in \mathbb{R}^{d\times d}$ and $f$ is a locally Lipschitz
continuous function satisfying that
$$f(0)=0,$$
\begin{equation}\label{1.3}
\lim_{r\longrightarrow 0}\ell_f\left( r\right)=0,
\end{equation}
where
\begin{equation}\label{1.4}
\ell_f\left( r\right) :=\sup_{x,y\in B_{%
\mathbb{R}
^{d}}\left( 0,r\right) }\frac{\left\Vert f\left( x\right) -f\left(
y\right) \right\Vert }{\left\Vert x-y\right\Vert }.
\end{equation}%
The asymptotic stability of the zero solution of the linear
Caputo-Katugampola fractional-order system:
\begin{equation}\label{1.3nb}
^{C}D_{t_0^{+}}^{\alpha ,\rho }x\left( t\right) =A x(t)
\end{equation}
is known to be equivalent to  the spectrum $\sigma(A)$ satisfies
$\sigma(A)\subset\{\lambda \in \mathbb{C}:\
|arg(\lambda)|>\frac{\alpha \pi}{2}|\}$ see \cite{3m}. It is to be
demonstrated that if the zero solution of \eqref{1.3nb} is
asymptotically stable then the zero
solution of \eqref{1.1s} is asymptotically stable which is our main result of theorem \ref{the3.1}.\\
For the proof of such theorem, two
main steps are presented. The first step consist on the
transformation of the linear part to a matrix which is "very
close" to a diagonal matrix. The second step describes the
construction of an appropriate Lyapunov-Perron operator whose aim
is to present a family of operators with the property that any
solution of the nonlinear system can be interpreted as a fixed
point. Finally, an application
is given to show the validity of the theoretical result.  In the special case $\rho=1$ Cong et al. \cite{8} developed the Lyapunov's first method for the zero solution of the Caputo fractional-order differential equations of order $\alpha \in (0,1)$.  Our results
presented in this paper mainly extend the work of Cong et al. \cite{8}.\\
The rest of the paper is organized as follows. Useful results in
relation with the fractional-order calculus and other
preliminaries are presented in section $2$. Then, the main
contributions, dealing with the linearized asymptotic stability
for Caputo-Katugampola fractional differential equations, is given in
section $3$. Finally, an application to Caputo-Katugampola
fractional-order Lorenz system is presented in section $4$ to show
the efficiency of the proposed approach.

\section{Preliminaries}

In this section, let us revisit some basics of the fractional
calculus. We adopt the notations of the Caputo-Katugampola fractional
integral and derivative from \cite{39,40,37,38}.

\begin{definition}
\label{def2.1} (Katugampola fractional integral)
Given $\alpha>0$, $\rho>0$ and an interval $[a,b]$ of $\mathbb{R}$, where $0<a<b$. The Katugampola fractional integral of a function $x\in L^1([a,b])$ is defined by
\begin{equation*}
I_{a^{+}}^{\alpha ,\rho }u\left( t\right) =\frac{\rho ^{1-\alpha
}}{\Gamma \left( \alpha \right) }\int\limits_{a}^{t}\frac{s^{\rho
-1}u\left( s\right) }{\left( t^{\rho }-s^{\rho }\right) ^{1-\alpha
}}ds,
\end{equation*}
where $\Gamma$ is the Gamma function.
\end{definition}

\begin{definition}
\label{def2.23219} (Katugampola fractional derivative) Given $0<\alpha<1$, $\rho>0$ and an interval $[a,b]$ of $\mathbb{R}$, where $0<a<b$. The
Katugampola fractional derivative is defined by
\begin{equation*}
D_{a^{+}}^{\alpha ,\rho }u\left( t\right) =\frac{\rho ^{\alpha
}}{\Gamma \left( 1-\alpha \right)
} t^{1-\rho} \displaystyle\frac{d}{dt}\int\limits_{a}^{t}\frac{s^{\rho -1}u\left( s\right)
}{\left( t^{\rho }-s^{\rho }\right) ^{\alpha }}ds.
\end{equation*}
\end{definition}

\begin{definition}
\label{def2.2} (Caputo-Katugampola fractional derivative) Given $0<\alpha<1$, $\rho>0$ and an interval $[a,b]$ of $\mathbb{R}$, where $0<a<b$. The
Caputo-Katugampola fractional derivative is defined by
\begin{align}
^{C}D_{a^{+}}^{\alpha ,\rho }u\left( t\right) =&D_{a^{+}}^{\alpha ,\rho }[u\left( t\right)-u\left( a\right)]\nonumber\\
&=\frac{\rho ^{\alpha
}}{\Gamma \left( 1-\alpha \right)
} t^{1-\rho} \displaystyle\frac{d}{dt}\int\limits_{a}^{t}\frac{s^{\rho -1}[u\left( s\right)-u(a)]
}{\left( t^{\rho }-s^{\rho }\right) ^{\alpha }}ds.\nonumber
\end{align}
\end{definition}

\begin{lemma}
\label{lem2.100} Let $h:[t_0,+\infty)\rightarrow \mathbb{R} $ be a
continuous function. Then, the semigroup property holds
\begin{equation*}
I_{a^{+}}^{\alpha ,\rho }I_{a^{+}}^{\beta
,\rho}u(t)=I_{a^{+}}^{\alpha+\beta ,\rho}u(t) ,\text{ \ }0<\alpha
,\ 0<\beta,\ 0<\rho .
\end{equation*}%
\end{lemma}

\begin{lemma}
\label{lem2.2}The Katugampola fractional integral of \ $u\left(
t\right) =\left( \frac{t^{\rho }-a^{\rho }}{\rho }\right) ^{\beta
}$ has the result
\begin{equation*}
\frac{\Gamma \left( \beta +1\right) }{\Gamma \left( \alpha +\beta +1\right) }%
\left( \frac{t^{\rho }-a^{\rho }}{\rho }\right) ^{\beta +\alpha
},\text{ \ \ }-1<\beta ,\ 0<\alpha,\  0<\rho .
\end{equation*}
\end{lemma}

\begin{lemma}
\label{lem2.3} If u is a constant, then the fractional derivative of $u$ is $%
^{C}D_{a^{+}}^{\alpha ,\rho }u\left( t\right) =0$.
\end{lemma}
\noindent Since $f$ is a locally Lipschitz continuous function,
then we have the existence and uniqueness of solutions of initial
value problems (\ref{1.1s}) (see \cite{37}).Let $x(t,x_0)$ denote
the solution of (\ref{1.1s}) on its maximal interval of existence
$I = [t_0, t_{max}(x_0))$ with $t_0 < t_{max}(x_0) \leq \infty$.

\begin{definition} \label{def2.3} The equilibrium point $x^*=0$ of (\ref{1.1s}) is called:\newline $\bullet $ stable if for any
$\varepsilon >0$ there is a $\delta =\delta \left( \varepsilon
\right) >0$ such that for every $\left\Vert x_{0}\right\Vert \leq
$ $\delta $ we have $t_{max}(x_0)=\infty$ and
\begin{equation*}
\left\Vert x\left( t,x_0 \right) \right\Vert \leq \varepsilon \text{ \ for all }%
t\geq t_0.
\end{equation*}%
\newline
 $\bullet $asymptotically stable if it is stable and, furthermore, there exists $c>0 $ such that
for every $\left\Vert x_{0}\right\Vert \leq $ $c$ we have $$%
\displaystyle\lim_{t\rightarrow \infty }x\left( t,x_0 \right)
=0.$$\newline
\end{definition}

\begin{definition}
The Mittag-Leffler function with two parameters is defined as
$$E_{\alpha,\beta}(z)=\displaystyle\sum_{k=0}^{+\infty} \frac{z^k}{\Gamma(k \alpha +\beta)},$$
where $\alpha>0$, $\beta>0$, $z\in \mathbb{C}$.\\
When $\beta=1$, one has $E_{\alpha}(z)=E_{\alpha,1}(z)$.
\end{definition}

\begin{proposition}
\label{pro2.1}\cite{8} Let $\lambda $ be an arbitrary complex
number with $\frac{\alpha \pi }{2}<\left\vert \arg \lambda
\right\vert \leq \pi .$ Then, the following statements
hold:\newline $\left( i\right) $ There exist a positive constant
$M(\alpha ,\lambda )$ and a positive number $t_{1}$\ such that
\begin{equation*}
\left\vert t^{\alpha -1}E_{\alpha ,\alpha }\left( \lambda
t^{\alpha }\right) \right\vert \leq \frac{M(\alpha ,\lambda
)}{t^{\alpha +1}}\text{ \ \ for any }t>t_{1}.
\end{equation*}%
$\left( ii\right) $ There exists a positive constant $C(\alpha
,\lambda )$
such that%
\begin{equation*}
\sup_{t\geq 0}\int\limits_{0}^{t}\left\vert \left( t-s\right)
^{\alpha -1}E_{\alpha ,\alpha }\left( \lambda \left( t-s\right)
^{\alpha }\right) \right\vert ds\leq C(\alpha ,\lambda ).
\end{equation*}
\end{proposition}

\begin{remark}
\label{pro2.2} Let $\lambda $ be an arbitrary complex number with $\frac{%
\alpha \pi }{2}<\left\vert \arg \lambda \right\vert \leq \pi.$
Then
\begin{equation*}
\sup_{t\geq a}\int\limits_{a}^{t}\left\vert \left( \frac{t^{\rho }-s^{\rho }%
}{\rho }\right) ^{\alpha -1}E_{\alpha ,\alpha }\left( \lambda \left( \frac{%
t^{\rho }-s^{\rho }}{\rho }\right) ^{\alpha }\right) s^{\rho
-1}\right\vert ds\leq C(\alpha ,\lambda ).
\end{equation*}
\end{remark}
Indeed, using the change of variable $u=\frac{s^{\rho }}{\rho }$,
we obtain
\begin{eqnarray*}
&&\int\limits_{a}^{t}\left\vert \left( \frac{t^{\rho }-s^{\rho }}{\rho }%
\right) ^{\alpha -1}E_{\alpha ,\alpha }\left( \lambda \left(
\frac{t^{\rho
}-s^{\rho }}{\rho }\right) ^{\alpha }\right) s^{\rho -1}\right\vert ds \\
&=&\int\limits_{\frac{a^{\rho }}{\rho }}^{\frac{t^{\rho }}{%
\rho }}\left\vert \left( \frac{t^{\rho }}{\rho }-u\right) ^{\alpha
-1}E_{\alpha ,\alpha }\left( \lambda \left( \frac{t^{\rho }}{\rho
}-u\right) ^{\alpha }\right) \right\vert du
\end{eqnarray*}
It follows from Proposition \ref{pro2.1}, that
$$\sup_{t\geq a}\int\limits_{a}^{t}\left\vert \left( \frac{t^{\rho }-s^{\rho }}{\rho }%
\right) ^{\alpha -1}E_{\alpha ,\alpha }\left( \lambda \left(
\frac{t^{\rho }-s^{\rho }}{\rho }\right) ^{\alpha }\right) s^{\rho
-1}\right\vert ds \leq C(\alpha ,\lambda ).$$

\noindent  We consider the following fractional differential
equation
\begin{equation*}
^{C}D_{a^{+}}^{\alpha ,\rho }u\left( t\right) =\lambda u\left(
t\right) +h\left( t\right) ,u\left( a\right) =c,0<\alpha
<1,0<\rho,
\end{equation*}%
where $h$ is a continuous function on $[a,+\infty[$.\\
The equation can be transferred to an equivalent integral one as \
\begin{equation*}
u\left( t\right) =u\left( a\right) +\lambda I_{a^{+}}^{\alpha
,\rho }u\left( t\right) +I_{a^{+}}^{\alpha ,\rho }h\left( t\right)
,\text{ }u\left( a\right) =c.
\end{equation*}%
To obtain an explicit clear solution, we apply the method of
successive approximation. Set $u_{0}=u\left( a\right) =c$ and
\begin{equation*}
u_{n+1}=u_{0}+\lambda I_{a^{+}}^{\alpha ,\rho
}u_{n}+I_{a^{+}}^{\alpha ,\rho }h\left( t\right) ,\text{ }0\leq n,
\end{equation*}%
\newline It follows that%
\begin{eqnarray*}
u_{1} &=&u_{0}+\lambda I_{a^{+}}^{\alpha ,\rho
}u_{0}+I_{a^{+}}^{\alpha
,\rho }h\left( t\right) \\
&=&u_{0}+\frac{u_{0}\lambda }{\Gamma \left( \alpha +1\right) }\left( \frac{%
t^{\rho }-a^{\rho }}{\rho }\right) ^{\alpha }+I_{a^{+}}^{\alpha
,\rho
}h\left( t\right) \\
&=&\sum_{k=0}^{1}\frac{u_{0}\lambda^k }{\Gamma \left( k\alpha +1\right) }%
\left( \frac{t^{\rho }-a^{\rho }}{\rho }\right) ^{k\alpha
}+I_{a^{+}}^{\alpha ,\rho }h\left( t\right) ,
\end{eqnarray*}%
\begin{equation*}
u_{2}=\sum_{k=0}^{2}\frac{u_{0}\lambda^k}{\Gamma \left( k\alpha +1\right) }%
\left( \frac{t^{\rho }-a^{\rho }}{\rho }\right) ^{k\alpha
}+\lambda I_{a^{+}}^{2\alpha ,\rho }h\left( t\right)
+I_{a^{+}}^{\alpha ,\rho }h\left( t\right) ,
\end{equation*}%
and

\begin{eqnarray*}
u_{n} &=&u_{0}\sum_{k=0}^{n}\frac{\lambda^k}{\Gamma \left( k\alpha +1\right) }%
\left( \frac{t^{\rho }-a^{\rho }}{\rho }\right) ^{k\alpha
}+\sum_{k=0}^{n-1}\lambda ^{k}I_{a^{+}}^{\alpha \left( k+1\right)
,\rho
}h\left( t\right) \\
&=&u_{0}\sum_{k=0}^{n}\frac{\lambda^k}{\Gamma \left( k\alpha +1\right) }%
\left( \frac{t^{\rho }-a^{\rho }}{\rho }\right) ^{k\alpha } \\
&&+\int\limits_{a}^{t}\sum_{k=0}^{n-1}\lambda ^{k}\frac{\rho
^{1-\alpha
\left( k+1\right) }}{\Gamma \left( \alpha k+\alpha \right) }\frac{s^{\rho -1}%
}{\left( t^{\rho }-s^{\rho }\right) ^{1-\alpha \left( k+1\right)
}}h\left(
s\right) ds \\
&=&u_{0}\sum_{k=0}^{n}\frac{\lambda^k}{\Gamma \left( k\alpha +1\right) }%
\left( \frac{t^{\rho }-a^{\rho }}{\rho }\right) ^{k\alpha } \\
&&+\int\limits_{a}^{t}\sum_{k=0}^{n-1}\frac{\lambda ^{k}}{\Gamma
\left(
\alpha k+\alpha \right) }s^{\rho -1}\left( \frac{t^{\rho }-s^{\rho }}{\rho }%
\right) ^{\alpha \left( k+1\right) -1}h\left( s\right) ds \\
&=&u_{0}\sum_{k=0}^{n}\frac{\lambda^k}{\Gamma \left( k\alpha +1\right) }%
\left( \frac{t^{\rho }-a^{\rho }}{\rho }\right) ^{k\alpha } \\
&&+\int\limits_{a}^{t}\left( \frac{t^{\rho }-s^{\rho }}{\rho
}\right) ^{\alpha -1}s^{\rho -1}\sum_{k=0}^{n-1}\frac{\lambda
^{k}}{\Gamma \left( \alpha k+\alpha \right) }\left( \frac{t^{\rho
}-s^{\rho }}{\rho }\right) ^{\alpha k}h\left( s\right) ds.
\end{eqnarray*}%
For $n\rightarrow \infty ,$ we obtain
\begin{eqnarray*}
u\left( t\right) &=&u_{0}E_{\alpha }\left( \lambda \left(
\frac{t^{\rho
}-a^{\rho }}{\rho }\right) ^{\alpha }\right) \\
&&+\int\limits_{a}^{t}\left( \frac{t^{\rho }-s^{\rho }}{\rho
}\right) ^{\alpha -1}E_{\alpha ,\alpha }\left( \lambda \left(
\frac{t^{\rho }-s^{\rho }}{\rho }\right) ^{\alpha }\right) s^{\rho
-1}h\left( s\right) ds.
\end{eqnarray*}%
Using Theorem 6.37, pp. 146 in \cite{Shilov} , there exists a
nonsingular matrix $T\in
\mathbb{C}
^{d\times d}$ transforming $A$ into the Jordan normal form, i.e.,
\begin{equation*}
T^{-1}AT=diag\left( A_{1},A_{2},...,A_{n}\right) ,
\end{equation*}%
where for $i=1,2,...,n$,
\begin{equation*}
A_{i}=\lambda _{i}id_{d_{i}\times d_{i}}+\eta _{i}N_{d_{i}\times
d_{i}},
\end{equation*}%
where $\eta _{i}\in \left\{ 0,1\right\} ,$ $\lambda _{i}\in $
$\left\{
\widehat{\lambda _{1}},\widehat{\lambda _{2}},...,\widehat{\lambda _{m}}%
\right\} ,$ and the nilpotent matrix $N_{d_{i}\times d_{i}}$ is
given by

\begin{equation*}
N_{d_{i}\times d_{i}}:=%
\begin{pmatrix}
0 & 1 & 0 &  &  & 0 & 0 \\
0 & 0 & 1 &  &  & 0 & 0 \\
. & . & . & . &  & . & . \\
. & . &  & . & . & . & . \\
. & . &  &  & . & . & . \\
0 & 0 &  &  &  & 0 & 1 \\
0 & 0 &  &  &  & 0 & 0%
\end{pmatrix}%
_{_{d_{i}\times d_{i}}}
\end{equation*}%
Let $\delta>0 $. Using the transformation $P_{i}=diag\left(
1,\delta ,...,\delta ^{d_{i}-1}\right) $, we get
\begin{equation*}
P_{i}^{-1}A_{i}P_{i}=\lambda _{i}id_{d_{i}\times d_{i}}+\delta
_{i}N_{d_{i}\times d_{i}},
\end{equation*}%
$\delta _{i}\in \left\{ 0,\delta \right\} .$ Therefore, by the
transformation
$y=\left( TP\right) ^{-1}x$ system (\ref{1.1s}) becomes%
\begin{equation}
^{C}D_{t_0^{+}}^{\alpha ,\rho }y\left( t\right) =diag\left(
J_{1},J_{2},...,J_{n}\right) y\left( t\right) +h\left( y\left(
t\right) \right) ,  \label{3.2}
\end{equation}%
where $J_{i}=\lambda _{i}id_{d_{i\times d_{i}}}$ for
$i=1,2,...,n,$ and
\begin{equation}
h\left( y\right) =diag\left( \delta _{1}N_{d_{1}\times
d_{1}},\delta _{2}N_{d_{2}\times d_{2}},...,\delta
_{n}N_{d_{n}\times d_{n}}\right) y\left( t\right) +\left(
TP\right) ^{-1}f\left( TPy\right) .  \label{3.3}
\end{equation}

\begin{remark}\cite{8}
\label{rem3.1} Note that the map $x\rightarrow diag\left( \delta
_{1}N_{d_{1}\times d_{1}},\delta _{2}N_{d_{2}\times
d_{2}},...,\delta _{n}N_{d_{n}\times d_{n}}\right) x$ is a
Lipschitz continuous function with Lipschitz constant $\delta $.
Thus, by (\ref{1.3}), we have
\begin{equation*}
h\left( 0\right) =0,\text{ \ }\lim_{r\rightarrow 0}\ell_h\left(
r\right) =\left\{
\begin{array}{l}
\delta \text{ \ if there exists }\delta _{i}=\delta , \\
0\text{\ otherwise\ .}%
\end{array}%
\right.
\end{equation*}
\end{remark}

\begin{remark}\cite{8}
\label{rem3.2} The type of stability of the zero solution of equations (%
\ref{1.1s}) and (\ref{3.2}) are the same.
\end{remark}
\noindent We denote by $C_{\infty }\left( \left[ t_0,+\infty
\right),\mathbb{R}^d\right)$ the space of all continuous functions
$ \xi : \left[ t_0,+\infty \right)\longrightarrow \mathbb{R}^d$
such that
$$\|\xi\|_{\infty}=\sup_{t\geq t_0}\|\xi(t)\|<\infty.$$
\noindent For any $x=\left( x^{1},x^{2},...,x^{n}\right) \in
\mathbb{R}^{d}=\mathbb{R}^{d_{1}}\times... \times
\mathbb{R}^{d_{n}},$ we define the operator
$$ \mathcal{F}_{x}:C_{ \infty}\left( \left[ t_0,+\infty
\right) ,%
\mathbb{R}
^{d}\right) \rightarrow C_{\infty}\left( \left[ t_0,+\infty \right) ,%
\mathbb{R}
^{d}\right) $$ as follows:
\begin{equation*}
\left( \mathcal{F} _{x}\xi \right) \left( t\right) =\left( \left(
\mathcal{F} _{x}\xi \right) ^{1}\left( t\right) ,\left(
\mathcal{F} _{x}\xi \right) ^{2}\left( t\right) ,...,\left(
\mathcal{F} _{x}\xi \right) ^{n}\left( t\right) \right) \text{ \ \
\ \ \ for }t\in \left[ t_0,+\infty \right).
\end{equation*}%
Where for $i=1,2,...,n$%
\begin{eqnarray*}
\left( \mathcal{F} _{x}\xi \right) ^{i}\left( t\right)
&=&E_{\alpha }\left( \left( \frac{t^{\rho }-t_0^{\rho }}{\rho
}\right) ^{\alpha }J_{i}\right) x^{i}
\\
&&+\int\limits_{t_0}^{t}\left( \frac{t^{\rho }-s^{\rho }}{\rho
}\right) ^{\alpha -1}E_{\alpha ,\alpha }\left( \left(
\frac{t^{\rho }-s^{\rho }}{\rho }\right) ^{\alpha }J_{i}\right)
s^{\rho -1}h^{i}\left( \xi \left( s\right) \right) ds
\end{eqnarray*}%
\begin{remark}
In the next section we will verify that $\mathcal{F} _{x}$ is well
defined. \end{remark}
\begin{remark}
Let $\xi \in C_{\infty}\left( \left[ t_0,+\infty \right),%
\mathbb{R}
^{d}\right)$, Then $\xi $ is a solution of (\ref{3.2}) with $\xi \left( t_{0}\right) =x$ if and only if it is a fixed point of the operator $\mathcal{F}%
_{x}$.
\end{remark}
\section{Linearized stability of Caputo-Katugampola fractional-order systems}
In this section we study the asymptotic stability of the system
(\ref{1.1s}).
\newline
\begin{proposition}
\label{pro3.2} Assume that $\sigma(A)\subset\{\lambda \in
\mathbb{C}:\ |arg(\lambda)|>\frac{\alpha \pi}{2}|\}$. Then, there
exists a constant $C\left( \alpha ,\lambda \right) $ depending on
$\alpha $ and $\lambda :=\left( \lambda _{1},\lambda
_{2,}...,\lambda _{n}\right) $ such that for all $x \in
\mathbb{R}
^{d}$ and $\xi ,\widehat{\xi }\in C_{\infty }\left( \left[
t_0,+\infty \right)
,%
\mathbb{R}
^{d}\right) $, it holds%
\begin{equation}
\left\Vert \mathcal{F} _{x}\xi -\mathcal{F} _{x}\widehat{\xi
}\right\Vert _{\infty }\leq C\left( \alpha ,\lambda \right) \ell_h
\Big(max\left( \left\Vert \xi \right\Vert _{\infty },\left\Vert
\widehat{\xi }\right\Vert _{\infty }\right)\Big) \left\Vert \xi
-\widehat{\xi }\right\Vert _{\infty } \label{3.4}
\end{equation}
and
\begin{eqnarray}\label{conc2}
\left\Vert \mathcal{F}_{x}\xi \right\Vert _{\infty } &\leq
&\max_{1\leq i\leq n}\sup_{t\geq t_0}\Big|E_{\alpha }\left(
\lambda _{i}\left( \frac{t^{\rho }-t_0^{\rho }}{\rho }\right)
^{\alpha }\right)\Big| \left\Vert x\right\Vert \nonumber\\
&&+C\left( \alpha ,\lambda\right) \ell_h\left( r\right) \left\Vert
\xi \right\Vert _{\infty }.
\end{eqnarray}%
\end{proposition}

\begin{proof}
For $i=1,2,...,n$, we have
\begin{eqnarray*}
\left\Vert \left( \mathcal{F} _{x}\xi \right) ^{i}\left( t\right)
-\left( \mathcal{F} _{x}\widehat{\xi }\right) ^{i}\left( t\right)
\right\Vert  &\leq &\ell_h \Big(max\left( \left\Vert \xi
\right\Vert _{\infty },\left\Vert \widehat{\xi }\right\Vert
_{\infty }\right)\Big)
\left\Vert \xi -\widehat{\xi }\right\Vert _{\infty }\times \\
&&\int\limits_{t_0}^{t}\left( \frac{t^{\rho }-s^{\rho }}{\rho
}\right) ^{\alpha -1}\Big|E_{\alpha ,\alpha }\left(\lambda_i
\left( \frac{t^{\rho }-s^{\rho }}{\rho }\right) ^{\alpha
}\right)\Big| s^{\rho -1}ds.
\end{eqnarray*}%
It follows from Remark \ref{pro2.2}, that%
\begin{eqnarray*}
\left\Vert \left( \mathcal{F} _{x}\xi \right) ^{i}-\left( \mathcal{F} _{%
x}\widehat{\xi }\right) ^{i}\right\Vert _{\infty } &\leq &\ell_h
\Big(max\left( \left\Vert \xi \right\Vert _{\infty },\left\Vert
\widehat{\xi }\right\Vert _{\infty }\right)\Big) C\left( \alpha
,\lambda _{i}\right) \left\Vert \xi -\widehat{\xi }\right\Vert
_{\infty }.
\end{eqnarray*}
Thus
\begin{equation*}
\left\Vert \mathcal{F} _{x}\xi -\mathcal{F} _{x}\widehat{\xi
}\right\Vert _{\infty }\leq C\left( \alpha ,\lambda\right) \ell_h
\Big(max\left( \left\Vert \xi \right\Vert _{\infty },\left\Vert
\widehat{\xi }\right\Vert _{\infty }\right)\Big) \left\Vert \xi
-\widehat{\xi }\right\Vert _{\infty },
\end{equation*}%
where $$C\left( \alpha ,\lambda\right) =\max \left\{ C\left(
\alpha ,\lambda _{1}\right) ,C\left( \alpha ,\lambda _{2} \right)
,...,C\left( \alpha ,\lambda _{n} \right) \right\}. $$ On the
other hand, for $\xi=0$, We have
$$\left( \mathcal{F} _{x}0 \right) ^{i}\left( t\right)
=E_{\alpha }\left( \left( \frac{t^{\rho }-t_0^{\rho }}{\rho
}\right) ^{\alpha }J_{i}\right) x^{i}.$$ Then using (\ref{3.4}),
 we get (\ref{conc2}).
\end{proof}

\begin{remark} It follows from Proposition \ref{pro3.2} that $\mathcal{F} _{x}$ is well-defined.\end{remark}

\noindent Up to now, we have found that the Lyapunov-Perron
operator is well-defined and Lipschitz continuous. Note that the
Lipschitz constant $C\left( \alpha ,\lambda\right) $ is
independent of the constant $\delta $ which is hidden in the
coefficients of system (\ref{3.2}). From now on, we choose and fix
the constant $\delta $ as follows $\delta :=\frac{1}{2C\left(
\alpha ,\lambda\right) }$.
\begin{lemma}
\label{lem3.1} Take $r>0$ such that
\begin{equation}
q:=C\left( \alpha ,\lambda\right) \ell_h\left( r\right) <1,
\label{3.5}
\end{equation}%
 and set
\begin{equation}
r^{\ast }=\frac{r\left( 1-q\right) }{\max_{1\leq i\leq
n}\sup_{t\geq
t_0}\Big|E_{\alpha }\left( \lambda _{i}\left( \frac{t^{\rho }-t_0^{\rho }}{\rho }%
\right) ^{\alpha }\right)\Big| }.  \label{3.6}
\end{equation}%
\newline
Let $B_{C_{\infty }}\left( 0,r\right) :=\left\{ \xi \in C_{\infty }\left( %
\left[ t_0,+\infty \right) ,%
\mathbb{R}
^{d}\right) :\left\Vert \xi \right\Vert_\infty \leq r\right\} .$ Then, for any $%
x\in B_{\mathbb{R}^d}\left( 0,r^{\ast }\right) $ we have
$\mathcal{F}_{x}\left( B_{C_{\infty }}\left( 0,r\right) \right)
\subset B_{C_{\infty }}\left( 0,r\right) $ and
\begin{equation*}
\left\Vert \mathcal{F}_{x}\xi -\mathcal{F}_{x}\widehat{\xi
}\right\Vert _{\infty }\leq q\left\Vert \xi -\widehat{\xi
}\right\Vert _{\infty }\text{ \ \ for all }\xi ,\widehat{\xi }\in
B_{C_{\infty }}\left( 0,r\right) .
\end{equation*}
\end{lemma}

\begin{proof}
Let $x\in
\mathbb{R}
^{d}$ with $\left\Vert x\right\Vert \leq r^{\ast }.$ Let $\xi
\in $\ $B_{%
C
_{\infty}}\left( 0,r\right) $. It follows from (\ref{conc2})
\begin{eqnarray*}
\left\Vert \mathcal{F}_{x}\xi \right\Vert _{\infty } &\leq
&\max_{1\leq i\leq n}\sup_{t\geq t_0}\Big|E_{\alpha }\left(
\lambda _{i}\left( \frac{t^{\rho }-t_0^{\rho }}{\rho }\right)
^{\alpha }\right)\Big| \left\Vert x\right\Vert \\
&&+C\left( \alpha ,\lambda\right) \ell_h\left( r\right) \left\Vert
\xi \right\Vert _{\infty } \\
&\leq &\left( 1-q\right) r+qr=r,
\end{eqnarray*}%
then $\mathcal{F}_{x}\left( B_{C_{\infty }}\left( 0,r\right)
\right) \subset B_{C_{\infty }}\left( 0,r\right) .$\\
Then for any $x\in B_{\mathbb{R}^d}\left( 0,r^{\ast }\right) $ and
$\xi ,\widehat{\xi }\in $\ $B_{C_{\infty }}\left( 0,r\right) $, it
follows from (\ref{3.4}) and (\ref{3.5}) that
\begin{eqnarray*}
\left\Vert \mathcal{F}_{x}\xi -\mathcal{F}_{x}\widehat{\xi
}\right\Vert _{\infty } &\leq &C\left( \alpha ,\lambda \right)
\ell_h\left(
r\right) \left\Vert \xi -\widehat{\xi }\right\Vert _{\infty } \\
&\leq &q\left\Vert \xi -\widehat{\xi }\right\Vert _{\infty }.
\end{eqnarray*}%
The proof is completed.
\end{proof}
\noindent Consider system (\ref{1.1s}):
$$^{C}D_{t_0^{+}}^{\alpha ,\rho }x\left( t\right) =A x(t)+
f\left(x\left( t\right) \right),\ t \geq t_0.$$ We now state the
main result of this paper.
\begin{theorem}\label{the3.1}  Assume that $\sigma(A)\subset\{\lambda \in
\mathbb{C}:\ |arg(\lambda)|>\frac{\alpha \pi}{2}|\}$ and the
nonlinear term $f$ is a locally Lipschitz continuous function
satisfying (\ref{1.3}). Then, the zero solution of (\ref{1.1s}) is
asymptotically stable.
\end{theorem}

\begin{proof}
In reason of Remark \ref{rem3.2}, it is sufficient to prove the asymptotic stability for the zero solution of system (\ref{3.2}). To do that, let $r^{\ast }$\ be defined as in (\ref{3.6}). Let $x$ $\in B_{%
\mathbb{R}
^{d}}\left( 0,r^{\ast }\right) $. Using the Banach fixed point
theorem and Lemma \ref{lem3.1}, there exists a unique fixed point
$\xi \in $\ $B_{C_{\infty }}\left( 0,r\right) $ of
$\mathcal{F}_{x}.$\ This point is also a solution of (\ref{3.2})
with the initial condition $\xi \left( t_0\right) =x.$ The zero
solution $0$ is stable, since the initial value problem for
Equation (\ref{3.2}) has unique solution. To finish the proof of
the theorem, we have to demonstrate that the zero solution $0$ is
attractive. Let $x=\left(
x^{1},x^{2},...,x^{n}\right) \in B_{%
\mathbb{R}
^{d}}\left( 0,r^{\ast }\right)$ and $\xi \left( t\right) =\left(
\left( \xi \right) ^{1}\left( t\right) ,\left( \xi \right)
^{2}\left( t\right)
,...,\left( \xi \right) ^{n}\left( t\right) \right) $ the solution of (\ref%
{3.2}) which satisfies $\xi \left( t_0\right) =x$.\\
From Lemma \ref{lem3.1}, we see that $%
\left\Vert \xi \right\Vert _{\infty }\leq r.$ Let
$b:=\displaystyle\limsup_{t\rightarrow \infty }\left\Vert \xi
\left( t\right) \right\Vert $ thus $b\in \left[ 0,r\right] .$ Let
$\epsilon >0$. Then, there exists $T\left( \epsilon \right) >t_0$
such that
\begin{equation*}
\left\Vert \xi \left( t\right) \right\Vert <b+\epsilon \text{ \ \ for all }%
t\geq T\left( \epsilon \right).
\end{equation*}%
Using Proposition \ref{pro2.1} $(i)$, we get
\begin{eqnarray*}
&&\displaystyle\limsup_{t\rightarrow \infty
}\Big\|\int\limits_{t_0}^{T\left( \epsilon \right) }\left(
\frac{t^{\rho }-s^{\rho }}{\rho }\right) ^{\alpha -1}E_{\alpha
,\alpha }\left( \lambda _{i}\left( \frac{t^{\rho }-s^{\rho }}{\rho
}\right)
^{\alpha }\right) s^{\rho -1}h^{i}\left( \xi \left( s\right) \right) ds \Big\|\\
&\leq &\max_{t\in \left[ t_0,T\left( \epsilon \right) \right]
}\left\Vert h^{i}\left( \xi \left( t\right) \right) \right\Vert
\displaystyle\limsup_{t\rightarrow \infty
}\int\limits_{t_0}^{T\left( \epsilon \right) }\frac{M\left( \alpha
,\lambda _{i}\right) }{\left( \frac{t^{\rho }-s^{\rho }}{\rho
}\right)
^{\alpha +1}}s^{\rho -1}ds \\
&\leq &\max_{t\in \left[ t_0,T\left( \epsilon \right) \right]
}\left\Vert h^{i}\left( \xi \left( t\right) \right) \right\Vert
\displaystyle\limsup_{t\rightarrow \infty
}\int\limits_{t_0}^{T\left( \epsilon \right) }\frac{\rho ^{\alpha
+1}M\left( \alpha ,\lambda _{i}\right) }{\left( t^{\rho }-s^{\rho
}\right)
^{\alpha +1}}s^{\rho -1}ds \\
&\leq &\max_{t\in \left[ t_0 ,T\left( \epsilon \right) \right]
}\left\Vert h^{i}\left( \xi \left( t\right) \right) \right\Vert
\displaystyle\limsup_{t\rightarrow \infty }\int\limits_{t_0^{\rho }}^{T\left( \epsilon \right) ^{\rho }}\frac{%
\rho ^{\alpha}M\left( \alpha ,\lambda _{i}\right) }{\left( t^{\rho
}-u\right) ^{\alpha +1}}du=0.
\end{eqnarray*}%
Thus, it follows from the fact that $\xi ^{i}\left( t\right)
=\left( \mathcal{F}_{x}\xi \right) ^{i}\left( t\right) $ and
$\displaystyle\lim_{t\rightarrow \infty }\
E_{\alpha }\left( \lambda _{i}\left( \frac{t^{\rho }-t_0^{\rho }}{\rho }%
\right) ^{\alpha }\right) =0$ that
\begin{eqnarray*}
\displaystyle\limsup_{t\rightarrow \infty }\left\Vert \xi
^{i}\left( t\right) \right\Vert
&=&\displaystyle\limsup_{t\rightarrow \infty
}\Big\|\int\limits_{T\left( \epsilon \right) }^{t}\left(
\frac{t^{\rho }-s^{\rho }}{\rho }\right) ^{\alpha -1}E_{\alpha
,\alpha }\left( \left( \frac{t^{\rho }-s^{\rho }}{\rho }\right)
^{\alpha }\right) s^{\rho -1}h^{i}\left( \xi \left( s\right)
\right) ds \Big\|\\
&\leq &\ell_h\left( r\right) C\left( \alpha ,\lambda _{i} \right)
\left( b+\epsilon \right).
\end{eqnarray*}%
Therefore,
\begin{eqnarray*}
b &\leq &\max \left\{ \displaystyle\limsup_{t\rightarrow \infty
}\left\Vert \xi ^{1}\left( t\right) \right\Vert
,\displaystyle\limsup_{t\rightarrow \infty }\left\Vert \xi
^{2}\left( t\right) \right\Vert ,...,\displaystyle\limsup_{t\rightarrow \infty }\left\Vert \xi ^{n}\left( t\right) \right\Vert \right\}  \\
&\leq &\ell_h\left( r\right) C\left( \alpha ,\lambda \right)
\left( b+\epsilon \right) .
\end{eqnarray*}%
We tend $\epsilon$ to zero we find,
\begin{equation*}
b\leq \ell_h\left( r\right) C\left( \alpha ,\lambda  \right) b.
\end{equation*}%
From the assumption $\ell_h\left( r\right) C\left( \alpha ,\lambda
\right) <1$, we obtain that $b=0$. This ends the proof.
\end{proof}
\section{Application}
To illustrate the theoretical result and to show its
effectiveness, the Caputo-Katugampola fractional-order Lorenz system is
given as an example. Such system can be written as:
\begin{equation}
\left\{
\begin{array}{l}
^{C}D_{t_0^{+}}^{\alpha ,\rho }x_1\left( t\right) =a\left( x_{1}\left( t\right)-x_{2}\left( t\right)\right)  \\
^{C}D_{t_0^{+}}^{\alpha ,\rho }x_2\left( t\right) =bx_{1}\left( t\right)-x_{1}\left( t\right)x_{3}\left( t\right)-cx_{2}\left( t\right) \\
^{C}D_{t_0^{+}}^{\alpha ,\rho }x_3\left( t\right) =x_{1}\left( t\right)x_{2}\left( t\right)-dx_{3}\left( t\right)%
\end{array}%
\right.
\end{equation}%
where $a,b,c$ and $d$ are four parameters and $\alpha ,\rho $ are
the fractional-orders. It can be rewritten as
\begin{equation*}
^{C}D_{t_0^{+}}^{\alpha ,\rho }x\left( t\right) =Ax\left(
t\right)+g\left( x\left( t\right)\right) ,
\end{equation*}%
where $x=\left( x_{1},x_{2},x_{3}\right) ^{T},A=\left(
\begin{array}{ccc}
a & -a & 0 \\
b & -c & 0 \\
0 & 0 & -d%
\end{array}%
\right) $ and $g\left( x\right) =\left(
\begin{array}{c}
0 \\
-x_{1}x_{3} \\
x_{1}x_{2}%
\end{array}%
\right). $\newline Let consider $a=-8,b=26,c=-7,d=3,\alpha =0.9,$
and $\rho =1.2$\newline
We design the linear state feedback controller as $u=BKx$ and select :%
\begin{equation*}
B=\left(
\begin{array}{c}
0 \\
1 \\
0%
\end{array}%
\right) \text{ and }K=\left(
\begin{array}{ccc}
0 & -50 & 0%
\end{array}%
\right).
\end{equation*}
The eigenvalue of the matrix $A+BK$ are $\lambda_1=-2.8229$,
$\lambda_2=-48.1771$ and $\lambda_3=-3$ which make $\left\vert
\arg \left(\lambda_i\right) \right\vert >\frac{\alpha \pi }{2},\
i\in \{1,2,3\}.$ Thus, according to Theorem \ref{the3.1}, the zero
solution of the closed-loop system
\begin{equation*}
^{C}D_{t_0^{+}}^{\alpha ,\rho }x\left( t\right) =Ax\left(
t\right)+g\left( x\left( t\right)\right)+Bu(t) ,
\end{equation*}%
is asymptotically stable.

\end{document}